\documentclass[a4paper,11pt]{article}
\usepackage{latexsym}
\usepackage{amsmath}
\usepackage{amssymb}
\usepackage{enumerate}
\usepackage{theorem}
\usepackage{array}
\newtheorem{theorem}{Theorem}[section]
\newtheorem{lemma}[theorem]{Lemma}

\theorembodyfont{\rmfamily}

\newtheorem{notation}[theorem]{Notation}

\newtheorem{remark}[theorem]{Remark}

\title{Numbers of points of hypersurfaces without lines over finite fields
\footnote{Submitted to to the Fq11 proceedings.}}
\author{
Masaaki Homma
\thanks{Partially supported by Grant-in-Aid
for Scientific Research (24540056), JSPS.}
\\
Department of Mathematics and Physics\\
Kanagawa University\\
Hiratsuka 259-1293, Japan\\
homma@kanagawa-u.ac.jp
}
\date{}

\begin{document}
\maketitle
\begin{abstract}
We give an upper bound for the number of points of a hypersurface
 over a finite field that has no lines on,
in terms of the dimension,
the degree, and the number of the elements of the finite field.
\end{abstract}
\section{Introduction}
A few years ago, we proved the Sziklai bound for plane curves
in \cite{hom-kim2009, hom-kim2010a, hom-kim2010b}.
Let $C$ be a plane curve of degree $d$ over a finite field $\mathbb{F}_q$.
If $C$ has no $\mathbb{F}_q$-line as a component,
then the number of $\mathbb{F}_q$-points $N_q(C)$ of $C$ is bounded by
\begin{equation}\label{sziklai}
N_q(C) \leq (d-1)q +1
\end{equation}
except for the case $d=q=4$ and $C$ is projectively isomorphic
over $\mathbb{F}_4$ to
\[
K: (X+ Y+Z)^4 +(XY +YZ + ZX)^2
    + XYZ(X+Y+Z) =0.
\]
This bound is not bad. Actually,
for
$d = 2, \, \sqrt{q} +1 \, \text{(if $q$ is square)}, \,
q-1, \, q, \, q+1, \, q+2,$
there are nonsingular curves of degree $d$ over $\mathbb{F}_q$
that attain the upper bound.

We extended this bound for curves in higher dimensional projective spaces
\cite{hom2012}.
In this paper, we attempt another generalization of this bound.

We use the notation $\theta_q(s) = \frac{q^{s+1}-1}{q-1}$
for any $s \in \mathbb{Z}.$
If $s >0$, $\theta_q(s)$ is exactly the number of $\mathbb{F}_q$-points
of $\mathbb{P}^s$.
For $s \leq 0$, $\theta_q(0)=1$, $\theta_q(-1)= 0$,
$\theta_q(-2)= -\frac{1}{q}$, etc.
The identity
$\theta_q(s) = q^s + \theta_q(s-1)$ holds for any $s \in \mathbb{Z}$.

\begin{theorem}\label{maintheorem}
Suppose $ n \geq 2$.
Let $X$ be a hypersurface in $\mathbb{P}^n$
defined over $\mathbb{F}_q$ of degree $d$.
If $X$ does not contain any $\mathbb{F}_q$-lines,
then
\begin{equation}\label{maininequality}
N_q(X) \leq (d-1)(q^{n-1} +1)+(d-2) (\theta_q(n-3) -1)
\end{equation}
except for the case $n=2$, $d=q=4$ and
the curve is projectively isomorphic over $\mathbb{F}_4$
to $K$ above.
\end{theorem}
When $n=2$, this theorem agrees with \cite[Theorem 3.1]{hom-kim2010b},
and several plane curves achieve the upper bound in
(\ref{maininequality}) as mentioned above.

When $n=3$, an elliptic quadric surface has no $\mathbb{F}_q$-line
and attains the upper bound in (\ref{maininequality}).
As for details, see \cite[Chapter 5]{hir1979}
and/or
\cite[Chapter 15]{hir1985}.

\begin{notation}
For a variety $X$ over $\mathbb{F}_q$,
$X(\mathbb{F}_q)$ denotes the set of $\mathbb{F}_q$-points of $X$.
In particular, $\mathbb{P}^n(\mathbb{F}_q)$
is the $n$-dimensional finite projective space over $\mathbb{F}_q$.
The number of $X(\mathbb{F}_q)$ is denoted by $N_q(X)$.
$\Check{\mathbb{P}}^n(\mathbb{F}_q)$
denotes the set of hyperplanes of $\mathbb{P}^n$
defined over $\mathbb{F}_q$.
If $Y$ is a variety defined over a finite extension of $\mathbb{F}_q$,
the image of $Y$ under the $q$-Frobenius map is denoted by
$Y^{(q)}$.

The number of elements of a finite set $S$ is denoted by ${}^{\#}S$.
\end{notation}

\section{The first step}
We prove Theorem~\ref{maintheorem} by induction on $n$.
Let $X$ be a hypersurface
over $\mathbb{F}_q$ of degree $d$ in $\mathbb{P}^n$.
Let $X = \bigcup_{i} X_i$ be the decomposition into
$\mathbb{F}_q$-irreducible components, and $\deg X_i = d_i$.
If
\[
N_q(X_i) \leq (d_i-1)(q^{n-1} +1)+(d_i-2) (\theta_q(n-3) -1)
\]
holds for each $X_i$,
we have upper bound (\ref{maininequality}) for $X$,
because $d = \sum_i d_i$ and $N_q(X) \leq \sum_{i} N_q(X_i)$.
So we assume, a priori, that
{\em $X$ is irreducible over $\mathbb{F}_q$}.

Under the above circumstance,
the following lemma holds.
\begin{lemma}\label{inductionlemma}
If, for any $H \in \Check{\mathbb{P}}^n(\mathbb{F}_q)$,
\[
N_q(X \cap H) \leq (d-1)(q^{n-2}+1) + (d-2)(\theta_q(n-4)-1)
\]
holds, then 
\[
N_q(X) \leq (d-1)(q^{n-1}+1) + (d-2)(\theta_q(n-3)-1).
\]
\end{lemma}
To show the above, the following lemma is needed.
\begin{lemma}\label{lemmainfinitegeometry}
Let $S$ be a subset of $\mathbb{P}^n(\mathbb{F}_q)$.
If ${}^{\#}(S\cap H) \leq \delta$ for
any $H \in  \Check{\mathbb{P}}^n(\mathbb{F}_q)$,
then
\[
{}^{\#}S \leq (\delta -1)q +1
 + \lfloor \frac{\delta -1}{\theta_q(n-2)} \rfloor,
\]
where
$\lfloor \frac{\delta -1}{\theta_q(n-2)} \rfloor$
is the integer part of $\frac{\delta -1}{\theta_q(n-2)}$.
\end{lemma}
\begin{proof}
See \cite[Proposition 2.2]{hom2012}.
\end{proof}

\begin{proof}[Proof of Lemma~\ref{inductionlemma}]
Let
\[
\delta =  (d-1)(q^{n-2}+1) + (d-2)(\theta_q(n-4)-1).
\]
Then, from the assumption on $N_q(X\cap H)$
and Lemma~\ref{lemmainfinitegeometry},
\[
N_q(X) \leq  (\delta -1)q +1
 + \lfloor \frac{\delta -1}{\theta_q(n-2)} \rfloor.
\]
Since
\begin{align*}
 \delta -1 &= (d-2)( q^{n-2} + \theta_q(n-4)) + q^{n-2} \\
           &\leq (d-2)\theta_q(n-2) + q^{n-2},
\end{align*}
we have
\[
 \frac{\delta -1}{\theta_q(n-2)} \leq (d-2)+ \frac{q^{n-2}}{\theta_q(n-2)}.
\]
Since $q^{n-2} < \theta_q(n-2)$,
$ \lfloor \frac{\delta -1}{\theta_q(n-2)} \rfloor =d-2.$
Hence
\begin{align*}
(\delta -1)q& +1
 + \lfloor \frac{\delta -1}{\theta_q(n-2)} \rfloor \\
        & = \left( (d-1)( q^{n-2}+1) + (d-2)(\theta_q(n-4)-1) -1\right)q
             +1 + d-2 \\
        & = \left( (d-1)q^{n-2} + (d-2)\theta_q(n-4) \right)q + d-1\\
        & = (d-1)(q^{n-1} +1) + (d-2)(\theta_q(n-3)-1).
\end{align*}
This completes the proof.
\end{proof}

In the next section,
we will show the following theorem.
\begin{theorem}\label{dq4}
For an irreducible surface $S$ over $\mathbb{F}_4$ of degree $4$
in $\mathbb{P}^3$,
the bound {\rm (\ref{maininequality})} is valid.
\end{theorem}

Here we complete the proof of Theorem~\ref{maintheorem} under Theorem~\ref{dq4}.

\begin{proof}[Proof of Theorem~\ref{maintheorem}]
When $n=2$, the statement of the theorem is
the same as \cite[Theorem 3.1]{hom-kim2010b}.
Let us consider the case $n=3$,
that is, $X$ is a surface of degree $d$ in $\mathbb{P}^3$
which is irreducible over $\mathbb{F}_q$.
Then we can apply Lemma~\ref{inductionlemma}
for $X$ except the case $d=q=4$.
This exceptional case is just the case where we handle in Theorem~\ref{dq4}.
Therefore the induction on $n \geq 3$ works well by Lemma~\ref{inductionlemma}.
\end{proof}

\section{Surface over $\mathbb{F}_4$ of degree $4$ in $\mathbb{P}^3$}
The aim of this section is to prove Theorem~\ref{dq4}.
An explicit statement is as follows.
\begin{theorem}\label{explicitform}
Let $S$ be an irreducible surface over $\mathbb{F}_4$ of degree $4$
in $\mathbb{P}^3$, then $N_4(S) \leq 51$.
\end{theorem}

If any $\mathbb{F}_4$-plane section $S \cap H$  of $S$ 
is not $\mathbb{F}_4$-isomorphic
to the curve $K \subset H= \mathbb{P}^2$,
one can apply Lemma~\ref{inductionlemma} (2),
and get $N_4(S) \leq 51$.

We need a property of the plane curve $K$.
\begin{remark}\label{remarkonK}
If a plane curve over $\mathbb{F}_4$ of degree $4$
is projectively isomorphic to $K$ over $\mathbb{F}_4$,
then any $\mathbb{F}_4$-line of the plane meets
the curve at least one $\mathbb{F}_4$-point.
Indeed,
$K(\mathbb{F}_4) = 
\mathbb{P}^2(\mathbb{F}_4) \setminus \mathbb{P}^2(\mathbb{F}_2)$
(see \cite[Section~3]{hom-kim2009}).
\end{remark}

The next lemma is trivial,
but meaningful for the proof of Theorem~\ref{explicitform}
\begin{lemma}\label{overalgclosed}
Let $Y$ be a surface in $\mathbb{P}^3$ over an algebraically closed field.
Let $P \in Y$ and $H$ be a plane which is not a component of $Y$ such that
$H \ni P$.
\begin{enumerate}[\rm (1)]
 \item Suppose that $P$ is a nonsingular point of $Y$.
Then $P$ is a singular point of $Y \cap H$ if and only if
$H = T_PY$,
where $T_PY$ is the embedded tangent plane to $Y$ at $P$.
 \item If $P$ is a nonsingular point of $Y \cap H$, then
it is also a nonsingular point of $Y$.
\end{enumerate}
\end{lemma}

Now we return to our surface $S$.
\begin{lemma}
If there is a singular $\mathbb{F}_4$-point on $S$,
then $N_4(S) < 51$.
\end{lemma}
\begin{proof}
Let $P \in S $ be a singular $\mathbb{F}_4$-point,
and $\mathcal{L}_P$ the set of $\mathbb{F}_4$-lines
passing through $P$.
Then
\[
N_4(S) = \sum_{l \in \mathcal{L}_P} 
  {}^{\#} \left( S(\mathbb{F}_4) \cap l \setminus \{P\}\right) +1.
\]
Since the intersection multiplicity $i(S.l;P)$ of $S$ and $l$ at $P$ is
at least $2$, we have
${}^{\#} \left( S(\mathbb{F}_4) \cap l\setminus \{P\} \right) \leq 2$.
Hence
$N_4(S) \leq 2\theta_4(2) +1 =43$.
\end{proof}

By this lemma, we can assume additionally that each point of $S(\mathbb{F}_4)$
is a nonsingular point of $S$.

\begin{lemma}\label{mainlemma}
Let $S$ be an irreducible surface over $\mathbb{F}_4$ of degree $4$
in $\mathbb{P}^3$.
Suppose that each $\mathbb{F}_4$-point of $S$ is nonsingular.
Let $H$ be an $\mathbb{F}_4$-plane of $\mathbb{P}^3$,
and $t(H) ={}^{\#}\{ P \in S(\mathbb{F}_4) \mid H= T_PS \}.$
Then
\begin{enumerate}[\rm (1)]
 \item $0 \leq t(H) \leq 5$;
 \item
    \begin{enumerate}[\rm (i)]
      \item if $t(H)=0$, then ${}^{\#}(S \cap H(\mathbb{F}_4)) \leq 14;$
      \item if $t(H)=1$, then ${}^{\#}(S \cap H(\mathbb{F}_4)) \leq 11;$
      \item if $t(H)=2$, then ${}^{\#}(S \cap H(\mathbb{F}_4)) \leq 10;$
      \item if $t(H)=3$, then ${}^{\#}(S \cap H(\mathbb{F}_4)) \leq 8;$
      \item if $t(H)=4$, then ${}^{\#}(S \cap H(\mathbb{F}_4)) \leq 6;$
      \item if $t(H)=5$, then ${}^{\#}(S \cap H(\mathbb{F}_4)) =5.$
    \end{enumerate}
  \item $t(H)=5$ if and only if $S\cap H$ is a double conic
  and irreducible as a topological space.
\end{enumerate}
\end{lemma}
\begin{proof}
Since $\deg S \cap H =4$, $N_4( S \cap H ) \leq 14$ and
equality holds if and only if $ S \cap H$ is projectively isomorphic to $K$
over $\mathbb{F}_4$
by \cite[Theorem 3.1]{hom-kim2010b}.
From now on, we assume that $t(H)>0$.
By Lemma~\ref{overalgclosed} and the assumption that each point of
$S(\mathbb{F}_4)$ is nonsingular,
$t(H)$ coincides with
the number of singular $\mathbb{F}_4$-points of $S\cap H$.

(Case I) Suppose that $S \cap H$ is absolutely irreducible.
Since $S\cap H$ is of degree $4$ in $H = \mathbb{P}^2$,
the arithmetic genus of $S \cap H$ is $3$.
Hence the number of singular points is at most $3$.
Hence $t(H) \leq 3$.
\begin{itemize}
\item When $t(H) =1$, let $P_0 \in S \cap H(\mathbb{F}_4)$
be the singular point of the curve $S \cap H$.
Then, since $i(l. S\cap H; P_0) \geq 2$ for any $\mathbb{F}_4$-line $l$ on $H$,
 \begin{align*}
  {}^{\#} \left(S \cap H(\mathbb{F}_4) \right)
       &= \sum_{l \in \mathcal{L}_{P_0} \cap \Check{H}}
                 {}^{\#}\left(
                       l \cap  (S \cap H(\mathbb{F}_4)) \setminus \{ P_0\} 
                        \right) + 1 \\
       & \leq 5 \cdot 2 + 1 =11,
 \end{align*}
where $\mathcal{L}_{P_0} \cap \Check{H}$ is the set of $\mathbb{F}_4$-lines
of $H = \mathbb{P}^2$ passing through $P_0$, which consists of $5$ lines.
\item When $t(H)= 2$,
the $\mathbb{F}_4$-line passing through two singular points
does not meet other points of $(S \cap H(\mathbb{F}_4))$.
So we have
${}^{\#} \left(S \cap H(\mathbb{F}_4) \right) \leq 10$
by similar arguments to the above.
\item When $t(H)=3$, the normalization of $S \cap H$ at those three points
is $\mathbb{P}^1$ defined over $\mathbb{F}_4$.
Hence
${}^{\#} \left(S \cap H(\mathbb{F}_4) \right) \leq N_4(\mathbb{P}^1) +3 = 8.$
\end{itemize}

(Case II)
Suppose that $S \cap H$ is not absolutely irreducible,
but irreducible over $\mathbb{F}_4$, which is divided into two sub-cases.
\begin{itemize}
 \item[(II-1)] Let $S \cap H$ be a union of two absolutely irreducible conics
 that are conjugate over $\mathbb{F}_4$ each other.
 Then  $S \cap H(\mathbb{F}_4)$ is contained in the intersection
 of those two conics.
 Hence $t(H) \leq {}^{\#}(S \cap H(\mathbb{F}_4)) \leq 4.$
 \item[(II-2)] Let
 $S \cap H = l \cup l^{(4)} \cup l^{(4^2)} \cup l^{(4^3)}$,
 where $l$ is a line over $\mathbb{F}_{4^4}$  and not defined over
 a smaller field.
 Then
$S \cap H (\mathbb{F}_4) \subset l \cap l^{(4)} \cap l^{(4^2)} \cap l^{(4^3)}$.
Hence $t(H) \leq {}^{\#}(S \cap H(\mathbb{F}_4)) \leq 1.$
\end{itemize}

(Case III)
Suppose that $S \cap H$ is not irreducible over $\mathbb{F}_4$.
Since $S \cap H$ has no $\mathbb{F}_4$-line as a component,
$S \cap H = C_1 \cup C_2$,
where $C_i$ is a plane curve of degree $2$ which is irreducible
over  $\mathbb{F}_4$.
\begin{itemize}
 \item[(III-1)] If $C_1 = l \cup l^{(4)}$ and $C_2 = l' \cup l'^{(4)}$ for lines $l$ and $l'$ over $\mathbb{F}_{4^2}$ that are not defined over $\mathbb{F}_4$,
then $S \cap H(\mathbb{F}_4) \subset (l \cap l^{(4)}) \cup (l' \cap l'^{(4)}).$
Hence $t(H) \leq {}^{\#}(S \cap H(\mathbb{F}_4)) \leq 2$.
 \item[(III-2)] If $C_1 = l \cup l^{(4)}$ and $C_2$ is absolutely irreducible,
 then $S \cap H(\mathbb{F}_4) \subset (l \cap l^{(4)}) \cup C_2(\mathbb{F}_4)$.
 Hence $t(H)=1$, and
 \[
 {}^{\#}(S \cap H(\mathbb{F}_4)) =\left\{
    \begin{array}{ccc}
     6 & \mbox{\rm if } & (l \cap l^{(4)}) \not\in C_2\\
     5 & \mbox{\rm if } & (l \cap l^{(4)}) \in C_2.
    \end{array}
     \right.
 \]
 \item[(III-3)] If both $C_1$ and $C_2$ are absolutely irreducible,
 then we have the following list according the number of
 $C_1 \cap C_2(\mathbb{F}_4)$.
 In this case, the set of singular points in $S\cap H(\mathbb{F}_4)$
 is just $C_1\cap C_2(\mathbb{F}_4)$.
 \[
 \begin{array}{ccc}
 {}^{\#}\left( C_1\cap C_2(\mathbb{F}_4) \right) &
             \  t(H) \   &
                   {}^{\#}\left( S \cap H(\mathbb{F}_4) \right) \\
      \hline
  1 &  1  &  9 \\
  2 &  2  &  8 \\
  3 &  3  &  7 \\
  4 &  4  &  6 \\
 C_1 =C_2& 5 & 5
 \end{array}
 \]
\end{itemize}
From the above observations, we have all assertions.
 \end{proof}

\begin{proof}[Proof of Theorem~\ref{explicitform}]
Let $N= N_4(S)$ and
\[
n_i = {}^{\#}\{
         H \in \Check{\mathbb{P}}^3(\mathbb{F}_4) \mid t(H) =i
             \}.
\]
Note that $n_i=0$ if $i>5$ by Lemma~\ref{mainlemma} (1).
Hence
\[
n_0 = \theta_4(3) - (n_1 + n_2 + n_3 + n_4 +n_5).
\]
Recall that each point of $S(\mathbb{F}_4)$ is nonsingular.
Hence
\[
n_1 + 2 n_2 + 3 n_3 + 4 n_4 + 5 n_5 = N.
\]

First we show that $n_5$ can be assumed at most $1$.
Let $H_1$ and $H_2$ be $\mathbb{F}_4$-planes in $\mathbb{P}^3$
such that $t(H_1)=t(H_2)=5$.
Suppose $H_1 \cap H_2 \cap S(\mathbb{F}_4) \neq \emptyset .$
Choose a point $P \in H_1 \cap H_2 \cap S(\mathbb{F}_4).$
For $i=1$ and $2$,
since $H_i \cap S$ is a double conic, $P$ is a singular point
of $H_i \cap S$.
Hence $H_i =T_P(S)$ by Lemma~\ref{overalgclosed},
and hence $H_1=H_2$.
Suppose the contrary:
$H_1 \cap H_2 \cap S(\mathbb{F}_4) = \emptyset $
and $H_1 \neq H_2.$
Let $l$ be the $\mathbb{F}_4$-line $H_1 \cap H_2$, and
$H'_3, H'_4, H'_5$ the other three $\mathbb{F}_4$-planes
containing the line $l$.
Then, for $j = 3, 4$ and $5$,
$H'_j \cap S$ is not $\mathbb{F}_4$-isomorphic to $K$.
In fact, since $l$ is an $\mathbb{F}_4$-line on the plane $H'_j$
and $l \cap (H'_j \cap S(\mathbb{F}_4)) 
\subset l \cap S(\mathbb{F}_4)= \emptyset$,
$H'_j \cap S$ cannot be projectively isomorphic to $K$
over $\mathbb{F}_4$  by Remark~\ref{remarkonK}.
Hence
$
N_4(H'_j \cap S) \leq 13.
$
Therefore
\begin{align*}
 N_4(S) &= N_4(H_1 \cap S) + N_4(H_2 \cap S) 
                   + \sum_{j=3}^5 N_4(H'_j \cap S) \\
        & \leq 5 + 5 +3 \times 13 = 49,
\end{align*}
which means that the target inequality already holds in this case.

Consider the correspondence
\[
\mathcal{P} = \{
        (P, H) \in S(\mathbb{F}_4) \times \Check{\mathbb{P}}^3(\mathbb{F}_4)
           \mid P \in H
              \}
\]
with two projections
$\pi_1: \mathcal{P} \to S(\mathbb{F}_4)$
and
$\pi_2: \mathcal{P} \to \Check{\mathbb{P}}^3(\mathbb{F}_4)$.
By using $\pi_1$,
we have ${}^{\#}\mathcal{P}=N \theta_4(2).$
On the other hand,
by using $\pi_2$,
\begin{align*}
  {}^{\#}\mathcal{P}  = &\sum_{j=0}^{5} \
              \sum_{
               \begin{subarray}{c}
                H \  \text{with} \\
                t(H)=j
               \end{subarray}
                 }
                   {}^{\#}(H \cap S (\mathbb{F}_4)) \\
            \leq &14(\theta_4(3) - (n_1+ \dots +n_5))
              +11n_1 + 10n_2 +8n_3 + 6n_4 +5n_5 \\
              &\text{(by Lemma~\ref{mainlemma})} \\
           = & 14 \theta_4(3)
             -2(\sum_{j=1}^{5} j n_j) +n_5 -n_1 \\
           = & 14 \theta_4(3) -2N + n_5 -n_1\\
           \leq & 14 \theta_4(3) -2N + 1 \ \text{(because $n_5 \leq 1$)}.
 \end{align*}
 Therefore
 \[
 N(\theta_4(2)+2) \leq 14 \theta_4(3) +1 =1191,
 \]
 and then
 $N \leq \lfloor 51 +\frac{18}{23} \rfloor =51$.
 This completes the proof.
 \end{proof}


\end{document}